\documentclass[11pt]{amsart}
\UseRawInputEncoding
\usepackage{graphicx}
\usepackage{mathrsfs}
\usepackage{hyperref}
\newtheorem{thm}{Theorem}[section]
\newtheorem{cor}[thm]{Corollary}

\theoremstyle{definition}

\theoremstyle{remark}

\begin{document}

\title[Fixed point theorem]{Fixed point theorem for an infinite Toeplitz matrix and its extension to general infinite matrices}%
\author{Vyacheslav M. Abramov}%
\address{24 Sagan Drive, Cranbourne North, Victoria, 3977, Australia}%

\email{vabramov126@gmail.com}%

\begin{abstract}
In [V. M. Abramov, \emph{Bull. Aust. Math. Soc.} \textbf{104} (2021), 108--117] the fixed point equation for an infinite nonnegative particular Toeplitz matrix has been studied.
In the present paper, we provide an alternative proof  for the existence of a positive solution in general case. The presented proof is based on an application of a variant of the fixed point theorem of M. A. Krasnosel'skii.
The results are then extended for the equations with infinite matrices of a general type.
\end{abstract}
\maketitle

\section{Introduction}\label{S1}
Let $\boldsymbol{x}=A\boldsymbol{x}$, where $A$ is an infinite matrix with nonnegative entries, and $\boldsymbol{x}$ is an unknown vector-column, the entries of which are denoted $x_0, x_1, \ldots$. By positive solution of the aforementioned matrix equation we mean such a vector $\boldsymbol{x}$, the entries of which satisfy the condition $x_i\geq0$, $i=0,1,\ldots$ and $\sum_{i=0}^{\infty}x_i>0$. Fixed point matrix equations of the form $\boldsymbol{x}=A\boldsymbol{x}$ or $\boldsymbol{x}=A\boldsymbol{x}+\boldsymbol{f}$ have wide application in economics. They describe a
quantitative economic model for the interdependencies between different sectors of a
national economy or different regional economies \cite{L}.
All these models are typically studied under the assumption that $\|A\|<1$, and their analysis uses principle of contraction mapping and iterative numerical procedures \cite{Kr_et_al}. The detailed study of linear systems can also be found in \cite{KLS}. For the infinite matrices, normed sequence spaces and
related topics, the reader can refer to the textbook \cite{FB}.

However, in a large number of problems in the areas of stochastic processes and applied probability (e.g. \cite{A2, T1, T2}) it is required to study the solutions of fixed point matrix equations or convolution type recurrence relations, in which the assumption $\|A\|<1$ becomes insufficient. The aim of the present paper is to find the conditions for a quite general class of infinite matrices (not necessarily obeying $\|A\|<1$), under which the equation $\boldsymbol{x}=A\boldsymbol{x}$ has a positive solution. Our result is first demonstrated on the equation $\boldsymbol{x}=T\boldsymbol{x}$ with the infinite Toeplitz matrix with nonnegative entries. Then it is reformulated and proved for the equations with infinite matrices of general type, for which the basic details of the proof remain unchanged.

The main novelty of the present paper is an analysis of the case $\|A\|>1$ for linear operator equations that previously has not been considered in the literature. The fixed point equations for nonlinear expansive operators has been earlier considered in \cite{WW, XY} and other papers.

\smallskip
Consider the equation
\begin{equation}\label{7}
\boldsymbol{x}=T\boldsymbol{x},
\end{equation}
 where
\begin{equation}\label{8}
T=\left(\begin{array}{cccccccc}t_0 &t_{-1} &t_{-2} &\cdots &t_{-n} &t_{-n-1} &\cdots\\
t_1 &t_0 &t_{-1} &\cdots &t_{-n+1} &t_{-n} &\cdots\\
t_2 &t_1 &t_0 &\cdots &t_{-n+2} &t_{-n+1} &\cdots\\
\vdots &\vdots &\vdots &\cdots &\vdots &\vdots &\cdots\\
t_{n} &t_{n-1} &t_{n-2} &\cdots &t_0 &t_{-1} &\cdots\\
\vdots &\vdots &\vdots &\cdots &\vdots &\vdots &\ddots
\end{array}\right),
\end{equation}
is an infinite Toeplitz matrix with nonnegative entries. Equation \eqref{7} with the matrix \eqref{8} has been studied in \cite{A}. To initiate our study in the present paper, we need to recall the main theorem proved in \cite{A}.

\smallskip
Let $\tau_{-n}(z)=\sum_{i=0}^{\infty}t_{i-n}z^i$, $0\leq z\leq 1$.

\begin{thm}\label{T1}
Assume that $n=\max\{j: t_{-j}>0\}<\infty$, and
\begin{equation}\label{6}
\frac{\mathrm{d}}{\mathrm{d}z}\sqrt[n]{\tau_{-n}(z)} \ \text{increases}.
\end{equation}
\begin{enumerate}
\item [(i)] If $\sum_{i=0}^\infty t_{i-n}>1$, then all positive solutions (if any) are bounded and $\lim_{i\to\infty}x_i=0$.
\item [(ii)] If $\sum_{i=0}^\infty t_{i-n}=1$, then all positive solutions are bounded if and only if $\sum_{i=0}^{\infty}it_{i-n}<n$.
In the case $n=1$, if $\sum_{i=0}^{\infty}it_{i-1}<1$, then $\lim_{i\to\infty}x_i$ exists, and $\lim_{i\to\infty}x_i=\frac{x_0t_{-1}}{1-\sum_{i=1}^{\infty}it_{i-1}}$.
\item [(iii)] If $\sum_{i=0}^\infty t_{i-n}<1$, then any positive solution is unbounded.
\end{enumerate}
\end{thm}

We have the following important comments on this theorem.
\begin{enumerate}
\item [(C1)] The condition $n=\max\{j: t_{-j}>0\}<\infty$ means that we deal with
\begin{equation}\label{9}
T=\left(\begin{array}{ccccccccc}t_0 &t_{-1} &t_{-2} &\cdots &t_{-n} &0 &0 &\cdots\\
t_1 &t_0 &t_{-1} &\cdots &t_{-n+1} &t_{-n} &0 &\cdots\\
t_2 &t_1 &t_0 &\cdots &t_{-n+2} &t_{-n+1} &t_{-n} &\cdots\\
\vdots &\vdots &\vdots &\cdots &\vdots &\vdots &\vdots &\cdots\\
t_{n} &t_{n-1} &t_{n-2} &\cdots &t_0 &t_{-1} &t_{-2} &\cdots\\
\vdots &\vdots &\vdots &\cdots &\vdots &\vdots &\vdots &\ddots
\end{array}\right),
\end{equation}
$t_{-n}>0$.
\smallskip
\item [(C2)] Claim (i) of Theorem \ref{T1} is true without technical condition \eqref{6}, since from the infinite system of equations
\begin{equation}\label{0}
0=\underbrace{(t_0-1)x_j}_{\text{negative}}+\sum_{i=1}^{j}t_ix_{j-i}+\sum_{i=1}^{n}t_{-i}x_{j+i}, \quad j=0,1,\ldots,
\end{equation}
under the assumption $\sum_{i=0}^{\infty}t_{i-n}>1$, we arrive at $\sum_{i=0}^{\infty}x_i<\infty$. (In other words, the proof of Theorem \ref{T1}(i) does not require technical assumption \eqref{6}. The aforementioned technical assumption is used in the proof of Theorem \ref{T1}(ii) only.)
\smallskip
\item[(C3)] The statements of Theorem \ref{T1} admit the extension for $n$ increasing to infinity in a straightforward way.
\smallskip
\item[(C4)] If \eqref{7} with the matrix $T$ defined by \eqref{9} or \eqref{8} has a positive solution, then there are infinitely many different positive solutions in general. For instance, if $\boldsymbol{x}^*$ is a positive solution, then $c\boldsymbol{x}^*$ is also a positive solution for any $c>0$. However, if $x_0$, $x_1$,\ldots, $x_{n-1}$ are fixed, then the solution found by the recurrence relation is, of course, unique.
\smallskip
\item[(C5)] If $\boldsymbol{x}$ is a positive solution of \eqref{7}, then every $x_i$, $i=0,1,\ldots$, must be strictly positive. This follows directly from representation \eqref{0}, where every term $x_it_{-n}$ with $t_{-n}>0$ can be expressed explicitly via the other terms, the sum of which is positive. If one of $x_i$'s is set to zero, then it turns out that all $x_i$, $i=0,1,\ldots$, are equal to zero, and we arrive at the contradiction.
\smallskip \item[(C6)]
Comment (C5) remains true in the situation when the assumption $n=\max\{j: t_{-j}>0\}<\infty$ is not satisfied, since in this case we have the similar to \eqref{0} system of the equations:
    \[
    0=\underbrace{(t_0-1)x_j}_{\text{negative}}+\sum_{i=1}^{j}t_ix_{j-i}+\sum_{i=1}^{\infty}t_{-i}x_{j+i}, \quad j=0,1,\ldots,
    \]
\end{enumerate}

On the basis of comments (C2) and (C3) we have the following extension of Theorem \ref{T1}(i).
\begin{cor}\label{C1}
For equation \eqref{7} with the matrix $T$ defined by \eqref{8} the following statement is true.
Suppose that $\sum_{i=-\infty}^{\infty}t_i>1$. If a positive solution of \eqref{7} exists, then $\sum_{i=0}^{\infty}x_i<\infty$.
\end{cor}

Conditions under which a solution of \eqref{7} with the matrix defined by \eqref{9} exists has been discussed in \cite[Section 3.3]{A}. It was shown that a positive solution exists if $\sum_{i=0}^{\infty}t_i<1$. The proof provided there was  entirely straightforward and based on a number of case studies. It cannot admit further extensions for more general types of matrix. In the alternative proof given in this paper for a more general situation, we show that the aforementioned condition can be easily obtained from the fixed point theorem of Krasnosel'skii. The statement is then extended for a general class of infinite matrices.

\smallskip
Below we recall the fixed point theorem of Krasnosel'skii  (see \cite{Burton, Krasnoselskii, Smart}).

\begin{thm} (Krasnosel'skii \cite{Krasnoselskii}) \label{T2}
Let $\mathcal{M}$ be a closed convex nonempty subset of a Banach space ($S, \|\cdot\|$). Suppose that $A$ and $B$ map $\mathcal{M}$ into $S$ such that
\begin{enumerate}
\item [(i)] $A\boldsymbol{x}+B\boldsymbol{y}\in \mathcal{M}$ (\text{for all} \ $\boldsymbol{x}, \boldsymbol{y}\in\mathcal{M}$),
\item [(ii)] $A$ is continuous and $A\mathcal{M}$ is contained in a compact set,
\item [(iii)] $B$ is a contraction mapping with constant $\alpha<1$.
\end{enumerate}
Then there is a vector $\boldsymbol{y}\in\mathcal{M}$ with $A\boldsymbol{y}+B\boldsymbol{y}=\boldsymbol{y}$.
\end{thm}
The fixed point theorem of Krasnosel'skii was originated for the problems of nonlinear analysis, and all its applications are related to nonlinear problems (e.g. \cite{BF, LL, SS, SSB}). There is a number of interesting development of this theorem, where some conditions of the theorem have been relaxed (see \cite{LL, XY2, XY3}). An application of Krasnosel'skii fixed point theorem in the theory of fractional calculus for non-local fractional delay differential systems of order $1<r<2$ in Banach spaces has been provided in \cite{WVUN}.

The applications provided in the present paper are not standard and typical for this theorem. In addition, as it has been mentioned in \cite{Burton}, Krasnosel'skii's theorem in this form is hardly applicable, since the condition (i) is too stringent. We shall use another form of the theorem given in \cite[Theorem 2]{Burton}. Implicitly, the construction given in \cite{Burton} was earlier used by O'Regan \cite{OR}, from whom the required theorem was not formulated. Then Burton's theorem has been further developed in the various studies.

The formulation of Burton's theorem is as follows.

\begin{thm} (Burton \cite{Burton})\label{T5}
Let $\mathcal{M}$ be a closed, convex, and nonempty subset of a Banach space ($S, \|\cdot\|$). Suppose that $A: \mathcal{M}\to S$ and $B: S\to S$ such that
\begin{enumerate}
\item [(i)] $B$ is a contraction with constant $\alpha<1$,
\item [(ii)] $A$ is continuous and $A\mathcal{M}$ is contained in a compact subset of $S$,
\item [(iii)] $[\boldsymbol{x}=B\boldsymbol{x}+A\boldsymbol{y}, \boldsymbol{y}\in\mathcal{M}]$ implies $\boldsymbol{x}\in\mathcal{M}$
\end{enumerate}
Then there is a $\boldsymbol{y}\in\mathcal{M}$ with $A\boldsymbol{y}+B\boldsymbol{y}=\boldsymbol{y}$.
\end{thm}

The rest of this paper is structured in two sections. In Section \ref{S2}, we prove the main result for infinite Toeplitz matrices. In Section \ref{S3}, the result is extended to general matrices with strictly positive entries. In Section \ref{S4} we conclude the paper.

\section{Fixed point theorem for infinite Toeplitz matrix}\label{S2}
We prove the following theorem.

\begin{thm}\label{T3}
For the fixed point equation defined by \eqref{7} and \eqref{8} assume that $\sum_{i=0}^{\infty}t_i< 1$ and $\sum_{i=-\infty}^{\infty}t_i<\infty$. Then the fixed point equation \eqref{7} has a positive solution.
\end{thm}

\begin{proof}
We define the norm of the infinite matrix $T$ with nonnegative entries based on the known norms for finite-dimensional matrices given in \cite[pages 294--295]{HJ}. Specifically,
\[
\|T\|=\sup_{k\geq0}\sum_{i=-\infty}^{k}t_{i}=\lim_{k\to\infty}\sum_{i=-\infty}^{k}t_{i}=\sum_{i=-\infty}^{\infty}t_{i},
\]
where the index value $k\geq0$ is associated with the $k+1$st row of the matrix.

The Banach space $(S, \|\cdot\|)$ we deal with is the set of all absolutely convergent series of real numbers with $\ell^1$ norm, i.e. it is the set of sequences $\boldsymbol{s}=\{s_0, s_1,\ldots\}$ obeying $\sum_{i=0}^{\infty}|s_i|<\infty$. 
 The norm of a vector $\boldsymbol{x}$ is $\|\boldsymbol{x}\|=\sum_{i=0}^{\infty}|x_i|$.

Consider first the case $\|T\|\leq1$. Assume first that $\max\{n: t_{-n}>0\}<\infty$ (i.e we first consider equation \eqref{7} with the matrix $T$ given by \eqref{9}). The existence of the solution of the fixed point equation considered in \cite{A}, and the conclusion about it was based on a wrong derivation. The corrected version of the proof is given below.

Consider first the case $n=1$, in which $x_0$ is an arbitrary positive value. We have
\[
x_1=\frac{x_0}{t_{-1}}\geq x_0, \quad x_2=\frac{(1-t_0)x_1-t_1x_0}{t_{-1}}\geq\frac{(1-t_0-t_1)x_1}{t_{-1}}\geq x_0.
\]
By induction, we arrive at the inequality
\[
x_{N+1}\geq\frac{(1-t_0-t_1-\ldots-t_N)x_N}{t_{-1}}.
\]
So, $x_N\leq x_{N+1}$ for all $N\geq0$.

In the case $n\geq2$, the sequence $x_N$ is no longer monotone increasing. Nevertheless, a positive solution of \eqref{7} exists. Indeed, consider the first equality of \eqref{0} when $j=0$:
\[
(1-t_0)x_0=t_{-1}x_1+t_{-2}x_2+\ldots+t_{-n}x_n.
\]
Setting $x_0=x_1=\ldots=x_{n-1}>0$ and using the notation $x_{n-1}^*=x_{n-1}$ we obtain
\[
x_n=\frac{(1-t_0-t_{-1}-\ldots-t_{-n+1})x_{n-1}^*}{t_{-n}}
\]
giving us $x_n\geq x_{n-1}^*$.

From the second equality of \eqref{0} when $j=1$ we have
\[
(1-t_0)x_1=t_1x_0+t_{-1}x_2+\ldots+t_{-n}x_{n+1}.
\]
It yields
\[
x_{n+1}=\frac{(1-t_1-t_0-t_{-1}-\ldots-t_{-n+2})x_{n-1}^*+t_{-n+1}x_n}{t_{-n}}
\]
Denoting $(1-t_1-t_0-t_{-1}-\ldots-t_{-n+2})x_{n-1}+t_{-n+1}x_n=x_{n}^*(1-t_1-t_0-t_{-1}-\ldots-t_{-n+1})$, we obtain
\[
x_{n+1}=\frac{(1-t_1-t_0-t_{-1}-\ldots-t_{-n+1})x_n^*}{t_{-n}}
\]
giving us $x_{n+1}\geq x_{n}^*$. The procedure continues, and for $x_{n+k}$ we obtain
\[
x_{n+k}=\frac{(1-t_k-t_{k-1}-\ldots-t_{-n+1})x_{n+k-1}^*}{t_{-n}}
\]
giving us $x_{n+k}\geq x_{n+k-1}^*$.
It follows from this procedure that all $x_n$, $n=0,1,\ldots$ are positive, which means that a positive solution of \eqref{7} exists.

Assume now that $\max\{n: t_{-n}>0\}<\infty$ is not satisfied. Now we prove that under the assumption $\sum_{i=-\infty}^{\infty}t_i\leq1$ a positive solution exists. Assume first that $\max\{n: t_{-n}>0\}=N$, where $N$ is a sufficiently large integer number. Following the arguments provided above, under this assumption a positive solution of \eqref{7} exists. Due to Comment (C4), it can be reckoned that the first entry of a solution $x_0$ is fixed and, say, equal to $1$. Then in a series of operator equations under different values of parameter $N$, the first entry $x_0$ of the vector of a positive solution can be set to $1$. From this we arrive at the conclusion that the first entry $x_0$ of the limiting solution, as $N\to\infty$, is equal to $1$ as well, and therefore this limiting solution is positive. Hence, in the case when $\max\{n: t_{-n}>0\}<\infty$ is not satisfied the existence of a positive solution is proved. Note that according to Comment (C6) all entries of a positive solution vector must be positive. Therefore all other entries $x_1$, $x_2$,\ldots of the vector $\boldsymbol{x}$ are positive too.

\smallskip
So, the rest of the proof is provided under the assumption that $1<\|T\|<\infty$.
From Corollary \ref{C1} we know that any positive solution $\boldsymbol{x}$ (if exists) satisfies the property $\lim_{i\to\infty}x_i=0$. From the proof of that theorem we have $\sum_{i=0}^{\infty}x_i<\infty$.
Based on this, in order to apply Theorem \ref{T5} we use the following construction. We represent the matrix $T$ as a sum $T_1+T_2$ of two matrices $T_1$ and $T_2$ defined by
\[
T_1=\left(\begin{array}{ccccc}t_0 &0 &0 &0 &\cdots\\
t_1 &t_0 &0 &0 &\cdots\\
t_2 &t_1 &t_0 &0 &\cdots\\
\vdots &\vdots &\vdots &\vdots &\ddots\end{array}\right),
\]
\[
 T_2=\left(\begin{array}{ccccccc}0 &t_{-1} &t_{-2} &\cdots &t_{-n} &t_{-n-1} &\cdots\\
0 &0 &t_{-1} &\cdots &t_{-n+1} &t_{-n} &\cdots\\
\vdots &\vdots &\vdots &\cdots &\vdots &\vdots &\ddots\end{array}
\right).
\]

Next, let $K\subset S$ be a set of all nonnegative vectors $\boldsymbol{s}$, $T_1: K\to K$, where $K$ is a cone of the aforementioned Banach space $(S, \|\cdot\|)$. We aimed to apply Theorem \ref{T5}. In our case,  $\mathcal{M}$ is the set of vectors $\boldsymbol{m}$ with positive entries having norm $\|\boldsymbol{m}\|=1$. Apparently, $\mathcal{M}$ is closed, convex and nonempty set. (The convexity follows, since for any $\boldsymbol{m}_1$, $\boldsymbol{m}_2\in\mathcal{M}$ and positive $\lambda_1$, $\lambda_2$ satisfying $\lambda_1+\lambda_2=1$, we have $\|\lambda_1\boldsymbol{m}_1+\lambda_2\boldsymbol{m}_2\|=\lambda_1\|\boldsymbol{m}_1\|+\lambda_2\|\boldsymbol{m}_2\|=1$.)
In addition, the set $T_2\mathcal{M}$ is compact. Indeed, following \cite[pp. 451--458]{Treves} a set $\mathcal{S}\subset S$ of elements having $\ell^{1}-norm$ to be compact should satisfy the following conditions. It should be bounded, closed and equismall at infinity. Recall (see \cite[pp. 451--458]{Treves}) that a set $\mathcal{S}$ is equismall at infinity, if for any $\epsilon>0$ there exists a natural number $n_{\epsilon}$ such that $\sum_{i=n_{\epsilon}}^{\infty}|s_i|<\epsilon$ for all $\boldsymbol{s}=(s_i)_{i=0}^\infty\in\mathcal{S}$. In our case, for any $\boldsymbol{m}\in\mathcal{M}$
\[
T_2\boldsymbol{m}=\left(\begin{array}{l}
\sum_{i=1}^{\infty}t_{-i}m_i\\
\sum_{i=1}^{\infty}t_{-i}m_{i+1}\\
\vdots\\
\sum_{i=1}^{\infty}t_{-i}m_{i+N}\\
\vdots
\end{array}\right).
\]
Since the set $\mathcal{M}$ consists of vectors $\boldsymbol{m}$, the entries of which are convergent positive sequences with $\|\boldsymbol{m}\|=1$, and $\sum_{i=1}^{\infty}t_{-i}<\infty$, then the problem reduces to show that for any $\epsilon>0$ and all $\boldsymbol{m}\in\mathcal{M}$ there exists $N$, such that
\[
\sum_{k=N}^{\infty}\sum_{i=1}^{\infty}t_{-i}m_{i+k}<\epsilon.
\]
The last inequality reduces to
\[
\sum_{k=N+1}^{\infty}m_{k}<\delta,
\]
for some positive small $\delta$. That is, the compactness of $T_2\mathcal{M}$ reduces to the compactness of $\mathcal{M}$.
Since $\mathcal{M}$ is known to be compact, then  $T_2\mathcal{M}$ is compact too.

Let $\boldsymbol{x}\in K$ be a vector. We have $\|T_1\boldsymbol{x}\|=\sum_{i=0}^{\infty}t_i\sum_{j=0}^{\infty}x_j=\alpha\|\boldsymbol{x}\|$, $\alpha<1$. Hence, $T_1$ is a contraction mapping.

So, (i) and (ii) of Theorem \ref{T5} are justified. Let us justify (iii).

It follows from Corollary \ref{C1} that if a positive solution of the equation \eqref{7} exists, then it must belong to $K$. According to Comment (C4) one can reckon that among the solutions of \eqref{7} there is a solution $\boldsymbol{x}^*$ obeying $\|\boldsymbol{x}^*\|=1$. Hence, an existing solution can be assumed to belong to $\mathcal{M}$. Consider the equation $\boldsymbol{x}=T_1\boldsymbol{x}+T_2\boldsymbol{y}$, $\boldsymbol{y}\in\mathcal{M}$. Show that if  $\boldsymbol{x}^{**}=\boldsymbol{x}^{**}(\boldsymbol{y})$ is a solution of this equation, then $\boldsymbol{x}^*$ is a solution of \eqref{7} belonging to $\mathcal{M}$ and vice versa. Indeed, suppose that
$\boldsymbol{x}^*-(T_1+T_2)\boldsymbol{x}^*=\boldsymbol{0}$, where $\boldsymbol{0}$ is the vector of zeros. Then, taking into account Comments (C5) and (C6), we arrive at the conclusion that there is a vector $\boldsymbol{\lambda}$ such that $\boldsymbol{y}=(\boldsymbol{\lambda}I)\boldsymbol{x}^*$, $I$ is the infinite unit matrix, and
\[
\boldsymbol{0}=(\boldsymbol{\lambda}I)\boldsymbol{x}^*-(\boldsymbol{\lambda}I)(T_1+T_2)\boldsymbol{x}^*=\boldsymbol{x}^{**}-T_1\boldsymbol{x}^{**}-T_2\boldsymbol{y}.
\]
Therefore, if $\boldsymbol{x}^*\in\mathcal{M}$, then $\boldsymbol{x}^{**}\in\mathcal{M}$. The inverse statement is also true, because of the one-to-one correspondence between $\boldsymbol{x}^*$ and $\boldsymbol{x}^{**}$ for any fixed $\boldsymbol{y}\in\mathcal{M}$.
\end{proof}

\section{Fixed point theorem for general infinity matrices}\label{S3}

Theorem \ref{T3} can be extended for any arbitrary matrix $T$ with nonnegative entries, the entries of which for convenience of the further formulations of the main result are denoted:
\[
T=\left(\begin{array}{cccc}t_0^{(1)} &t_{-1}^{(1)} &t_{-2}^{(1)} &\cdots\\
t_1^{(2)} &t_0^{(2)} &t_{-1}^{(2)} &\cdots\\
t_2^{(3)} &t_1^{(3)} &t_0^{(3)} &\cdots\\
\vdots &\vdots &\vdots &\ddots
\end{array}\right).
\]

For this matrix, the theorem below covers the case $\|T\|>1$ only.

\begin{thm}\label{T4}
Assume that $t_i^{(j)}>0$ for all $i=\ldots, -1,0,1,\ldots$ and $j=1,2,\ldots$.
Assume also that
\begin{align}
&\sum_{i=0}^{\infty}\sup_{k}t_{i}^{(k+1)}<1,\label{1}\\
&\sup_{k}\sum_{i=1}^{\infty}t_{-i}^{(k)}<\infty,\label{2}\\
&\liminf_{k\to\infty}\left[\sum_{i=0}^{k}t_{i}^{(k+1)}+\sum_{i=1}^{\infty}t_{-i}^{(k+1)}\right]>1,\label{3}
\end{align}
Then the equation $\boldsymbol{x}=T\boldsymbol{x}$ has a positive solution satisfying the property $\|\boldsymbol{x}\|<\infty$.
\end{thm}

\begin{proof}
To start the proof let us first discuss our assumptions. The assumption that the terms $t_i^{(j)}$ all are strictly positive, $i=\ldots,-1,0,1,\ldots$, $j=1,2,\ldots$, excludes the situation when any positive solution is impossible because of inconsistency of the left- and right-hand sides of the equations for the coordinates of the vector. For instance, if \eqref{1} is satisfied but $t_{-i}^{(j)}=0$ for all $i=1,2,\ldots$, then from the first equation we have $t_0^{(1)}x_0=x_0$ or $x_0=0$, and from the following equations of the recursion we obtain $x_n=0$ for all $n\geq0$. That is no positive solution exists.

Assumption \eqref{1} is a basic assumption. It guarantees that the matrix
\[
T_1=\left(\begin{array}{ccccc}t_0^{(1)} &0 &0 &0 &\cdots\\
t_1^{(2)} &t_0^{(2)} &0 &0 &\cdots\\
t_2^{(3)} &t_1^{(3)} &t_0^{(3)} &0 &\cdots\\
\vdots &\vdots &\vdots &\vdots &\ddots\end{array}\right)
\]
is a contraction mapping. The last is true since $\|T_1\boldsymbol{x}\|\leq\|\boldsymbol{x}\|\sum_{i=0}^{\infty}\sup_{j}t_i^{(j)}=\alpha\|\boldsymbol{x}\|$, $\alpha<1$.
The matrix $T_1$ is similar to that was used in the proof of Theorem \ref{T3}, when the matrix $T$ was presented as $T=T_1+T_2$.

\eqref{2} keeps the norm of matrix $T$ finite. Otherwise, a positive solution does not need to exist.

\eqref{3} is an important assumption that enables us to prove the following claim (*): if a solution of the equation $\boldsymbol{x}=T\boldsymbol{x}$ exists, then it satisfies the property $\sum_{i=0}^{\infty}x_i<\infty$, where $x_0$, $x_1$,\ldots are the coordinates of that solution. The proof of (*), as the most important one, is given below.

Prove (*). Suppose that $\boldsymbol{x}=\boldsymbol{x}^*$ is a positive fixed point solution.
Since $\|T\|>1$, then $\boldsymbol{x}^*$ must belong to some `good' subset $\mathcal{M}\subset\mathbb{R}_+^\infty$. Our aim is to characterize that subset and prove that $\|\boldsymbol{x}^*\|<\infty$.

For a fixed point solution $\boldsymbol{x}^*$, the fixed point equation in the matrix-operator form is as follows:
\begin{equation}\label{5}
\boldsymbol{x}^*=\left(\begin{array}{c}x_0^*\\ x_1^*\\ \vdots\end{array}\right)=\left(\begin{array}{cccc}t_0^{(1)} &t_{-1}^{(1)} &t_{-2}^{(1)} &\cdots\\
t_1^{(2)} &t_0^{(2)} &t_{-1}^{(2)} &\cdots\\
t_2^{(3)} &t_1^{(3)} &t_0^{(3)} &\cdots\\
\vdots &\vdots &\vdots &\ddots
\end{array}\right)\left(\begin{array}{c}x_0^*\\ x_1^*\\ \vdots\end{array}\right).
\end{equation}
Then, for the $j$th coordinate of $\boldsymbol{x}^*$ its expansion is
\begin{equation}\label{4}
\begin{aligned}
0=&\underbrace{(t_0^{(j+1)}-1)x_j^*}_{\text{negative value}}+\sum_{i=1}^{j}t_i^{(j+1)}x_{j-i}^*+\sum_{i=1}^{\infty}t_{-i}^{(j+1)}x_{j+i}.
\end{aligned}
\end{equation}
Assuming that $j$ in the equation is large enough and \eqref{3} is satisfied, i.e. $\sum_{i=-\infty}^{\infty}t_i^{(j+1)}>1$ for each $j$ and not approaching 1 as $j\to\infty$, we are aimed to prove that the tail of the series
\[
x^*_{j+1}+\ldots+x^*_{j+n}+\ldots
\]
is finite, and, hence, $\|\boldsymbol{x}^*\|<\infty$. That is, the `good' subset $\mathcal{M}$ is the set of all positive sequences $x_0$, $x_1$,\ldots such that $\sum_{i=0}^{\infty}x_i<\infty$.

Indeed, under condition \eqref{3}, consider the series of Toeplitz matrices $T^{(j)}$ given by

\begin{equation*}
T^{(j)}=\left(\begin{array}{ccccc}t_0^{(j+1)} &t_{-1}^{(j+1)} &t_{-2}^{(j+1)} &\cdots &t_{-j}^{(j+1)}\\
t_1^{(j+1)} &t_0^{(j+1)} &t_{-1}^{(j+1)} &\cdots &t_{1-j}^{(j+1)}\\
t_2^{(j+1)} &t_1^{(j+1)} &t_0^{(j+1)} &\cdots &t_{2-j}^{(j+1)}\\
\vdots &\vdots &\vdots &\vdots &\vdots\\
t_j^{(j+1)} &t_{j-1}^{(j+1)} &t_{j-2}^{(j+1)} &\cdots &t_0^{(j+1)}
\end{array}\right)
\end{equation*}
with different $j\geq j_0$, where $j_0$ is large.

For each of the systems $\boldsymbol{x}=T^{(j)}\boldsymbol{x}$, $j\geq j_0$, a solution $\boldsymbol{x}_j$ exists ($\boldsymbol{x}_j$ is $(j+1)$-dimensional vector). Its coordinates are asymptotically close to the corresponding coordinates of the solution of the fixed point equation
\begin{equation}\label{10}
\boldsymbol{x}=\left(\begin{array}{cccccccc}t_0^{(j+1)} &t_{-1}^{(j+1)} &t_{-2}^{(j+1)} &\cdots &t_{-j}^{(j+1)} &0 &0  &\cdots \\
t_1^{(j+1)} &t_0^{(j+1)} &t_{-1}^{(j+1)} &\cdots &t_{1-j}^{(j+1)} &t_{-j}^{(j+1)}  &0 &\cdots \\
t_2^{(j+1)} &t_1^{(j+1)} &t_0^{(j+1)} &\cdots &t_{2-j}^{(j+1)} &t_{1-j}^{(j+1)} &t_{-j}^{(j+1)} &\cdots \\
\vdots &\vdots &\vdots &\cdots &\vdots &\vdots &\vdots &\cdots  \\
t_j^{(j+1)} &t_{j-1}^{(j+1)} &t_{j-2}^{(j+1)} &\cdots &t_0^{(j+1)} &t_{-1}^{(j+1)} &t_{-2}^{(j+1)} &\cdots \\
0 &t_j^{(j+1)} &t_{j-1}^{(j+1)} &\cdots &t_{1}^{(j+1)} &t_{0}^{(j+1)} &t_{-1}^{(j+1)} &\cdots \\
0 &0 &t_j^{(j+1}) &\cdots &t_{2}^{(j+1)} &t_{1}^{(j+1)} &t_{0}^{(j+1)} &\cdots \\
\vdots &\vdots &\vdots &\cdots &\vdots &\vdots &\vdots &\ddots  \\
\end{array}\right)\boldsymbol{x}
\end{equation}
with infinite Toeplitz matrix. The last is true, since according to Theorem \ref{T1} and Comment (C2) the solution of the last equation satisfies the property $\sum_{i=0}^{\infty}x_i<\infty$, and if we denote the vector of corresponding solution of the equation $\boldsymbol{x}=T^{(j)}\boldsymbol{x}$ by $\boldsymbol{x}^{(j)}$, then the system of equations for its entries $x_0^{(j)}$, $x_1^{(j)}$,\ldots, $x_j^{(j)}$ is
\[
\begin{aligned}
0=&(t_0^{(j+1)}-1)x_k^{(j)}+\sum_{i=1}^{k}t_i^{(j+1)}x_{k-i}^{(j)}+\sum_{i=1}^{j-k}t_{-i}^{(j+1)}x_{k+i}^{(j)}, \quad k=0,1,\ldots,j.
\end{aligned}
\]
The similar system of equations for \eqref{10} is
\[
\begin{aligned}
0=&(t_0^{(j+1)}-1)x_k+\sum_{i=1}^{k}t_i^{(j+1)}x_{k-i}+\sum_{i=1}^{j}t_{-i}^{(j+1)}x_{k+i}, \quad k=0,1,\ldots,j.
\end{aligned}
\]
It is readily seen from these systems of equations that given that $\sum_{i=0}^{\infty}x_i<\infty$, under the appropriate normalization condition we obtain $x_{i}^{(j)}\to x_i$ as $j\to\infty$.
Hence, under the assumptions given in the theorem, we have $\limsup_{j\to\infty}\sum_{i=0}^{j}x_i^{(j)}<\infty$, where $x_0^{(j)}$, $x_1^{(j)}$,\ldots, $x_j^{(j)}$ are the entries of the vector-solution of the equation $\boldsymbol{x}=T^{(j)}\boldsymbol{x}$.

We shall now show that there exists a solution $\boldsymbol{x}^{*}$ of the original equation $\boldsymbol{x}=T\boldsymbol{x}$ obeying $\|\boldsymbol{x}^{*}\|<\infty$ as well.
Indeed, the required solution can be approached by the series of solutions of the equations $\boldsymbol{x}=\tilde T^{(j)}\boldsymbol{x}$, where
\begin{equation*}
\tilde T^{(j)}=\left(\begin{array}{ccccc}t_0^{(1)} &t_{-1}^{(1)} &t_{-2}^{(1)} &\cdots &t_{-j}^{(1)}\\
t_1^{(2)} &t_0^{(2)} &t_{-1}^{(2)} &\cdots &t_{1-j}^{(2)}\\
t_2^{(3)} &t_1^{(3)} &t_0^{(3)} &\cdots &t_{2-j}^{(3)}\\
\vdots &\vdots &\vdots &\vdots &\vdots\\
t_j^{(j+1)} &t_{j-1}^{(j+1)} &t_{j-2}^{(j+1)} &\cdots &t_0^{(j+1)}
\end{array}\right),
\end{equation*}
as $j$ increases to infinity. According to the convention, a solution of the original equation $\boldsymbol{x}=T\boldsymbol{x}$ exists. Therefore, for sufficiently large $j$, a solution of the equation $\boldsymbol{x}=\tilde T^{(j)}\boldsymbol{x}$ exists too, and the sequence of solutions $\tilde{\boldsymbol{x}}_j$ as $j\to\infty$ must approach a solution of the original equation. As well, any series of solutions $\tilde{\boldsymbol{x}}_j$ of the equation $\boldsymbol{x}=\tilde T^{(j)}\boldsymbol{x}$ obeys $\limsup_{j\to\infty}\|\tilde{\boldsymbol{x}}_j\|<\infty$, since the asymptotic behavior of the solution $\tilde{\boldsymbol{x}}_j$ is specified by the similar asymptotic behavior of the solutions $\boldsymbol{x}_j$ for large $j$. The required statement follows.

As mentioned above, under assumption \eqref{1} $T_1: K\to K$ is a contraction mapping.
The map $T_2: \mathcal{M}\to K$, being linear, is continuous. As in the proof of Theorem \ref{T3}, $\mathcal{M}$ is the set of all positive convergent sequences $\boldsymbol{m}$ obeying $\|\boldsymbol{m}\|=1$.
According to an assumption of the theorem, $t_{i}^{j}>0$ for all $i=\ldots,-1,0,1,\ldots$, $j=1,2,\ldots$. Therefore, it follows from presentations \eqref{5} and \eqref{4}, any positive solution $\boldsymbol{x}^*$ satisfies the property $x_j^*>0$, $j=0, 1,\ldots$. This property is an analogue of the properties given in Comments (C5) and (C6) that have been used in the proof of Theorem \ref{T3}. Hence, statement (iii) of Theorem \ref{T5} is justified similarly to that in the proof of Theorem \ref{T3}.
Thus we arrive at the conclusion that a positive solution of the fixed point equation $\boldsymbol{x}=T\boldsymbol{x}$ exists.
\end{proof}

\section{Concluding remarks}\label{S4}
The circle of research problems of the present paper and the earlier one \cite{A} of the author were originated from the applied problems of the theories of stochastic processes and applied probability mentioned in the introduction, where the certain recurrence relations of convolution type have been considered. The perspectives of the future research seem to be in further applications of the results obtained in these two papers to advanced telecommunication systems, that in turn may initiate novel studies of operator equations.

\section*{Statements and declarations}

No conflict of interest was reported by the author. There was no financial support for this research.

\subsection*{Acknowledgement}
The author thanks the reviewers for careful reading and relevant comments. As well, the author expresses his gratitude to M. Yumagulov, A. Mikhailov and all other people, who made critical comments officially or privately.

\end{document}